\documentclass{amsproc}
\usepackage{amsmath, tikz-cd}
\usepackage{multicol}
\usepackage{subfigure}
\usepackage{amssymb}
\usepackage{amsfonts}
\usepackage{enumitem}
\usepackage{amsthm}
\usepackage{xcolor}
\usepackage{mathtools}

\begin{document}
\title{Rhaly-Type Operators in Several Complex Variables}
 
\author{Michael R. Pilla}

\subjclass{47A13, 47A05, 47B32}

\keywords{Rhaly Operator, Ces\`aro Operator, Drury-Arveson Space}

\maketitle

\bibliographystyle{amsplain}

\numberwithin{equation}{section}
\theoremstyle{plain}
\newtheorem{Proposition}[equation]{Proposition}
\newtheorem{Corollary}[equation]{Corollary}
\newtheorem*{Corollary*}{Corollary}
\newtheorem{Theorem}[equation]{Theorem}
\newtheorem*{Theorem*}{Theorem}
\newtheorem{Lemma}[equation]{Lemma}
\theoremstyle{definition}
\newtheorem{Definition}[equation]{Definition}
\newtheorem{Conjecture}[equation]{Conjecture}
\newtheorem{Example}[equation]{Example}
\newtheorem{Exercise}[equation]{Exercise}
\newtheorem{Remark}[equation]{Remark}
\newtheorem{Question}[equation]{Question}

\setlength{\parskip}{12pt}
\setlength{\parindent}{0in}

\newcommand{\R}{\mathbb{R}}
\newcommand{\C}{\mathbb{C}}
\newcommand{\D}{\mathbb{D}}

\begin{abstract}
Recently, a generalization of the Ces\`aro operator to several variables was introduced as a tuple acting on the Drury-Arveson space \cite{P}. We determine the norm of these Ces\`aro-type operators and utilize it along with this Ces\`aro tuple definition to generalize the classical Rhaly operator and recover some basic facts about this operator when applied to the generalization.
\end{abstract}

\section{Introduction}

\subsection{The Rhaly Operator}

Rhaly operators, also known as terraced operators, are one well-studied way in which the celebrated Ces\`aro operator has been generalized \cite{Ma}, \cite{R}, \cite{R1}, \cite{R2}. Given a sequence of complex numbers $\{a_n\}$, the Rhaly operator is defined by 

\begin{equation}
R_{\{a_n\}}=
\begin{pmatrix}
a_1 & 0 & 0 & \cdots \\
a_2 & a_2 & 0 & \cdots \\
a_3 & a_3 & a_3 & \cdots\\
\vdots & \vdots & \vdots & \ddots.
\end{pmatrix}
\end{equation}

Setting $a_n=\frac{1}{n+1}$ gives the classical Ces\`aro operator \cite{Brown}, \cite{Ross}, \cite{Ross2}. Usually this operator is handled as acting on $\ell^2$ (and hence the Hardy space $H^2(\mathbb{D})$), although it can apply to any sequences space. Perhaps the most celebrated result due to Rhaly \cite{R2} is the fact that if $L=\lim_{n \rightarrow \infty}(n+1)|a_n|$ exists then the Rhaly operator is bounded. This gives a sufficient condition for determining when a given Rhaly operator is bounded. 

Rhaly operators have enjoyed numerous generalization approaches including to $\ell^p$ spaces \cite{Gala}, \cite{Gala2}, to weighted Hardy spaces \cite{Bell}, to matrices with more general entries \cite{L}, and to more general topological spaces \cite{Dogan}. Rhaly operators have yet to be generalized to reproducing kernel Hilbert spaces acting on several complex variables. Part of the reason for this is because they are a generalization of the classic Ces\`aro operator which had yet to be generalized on these spaces.

Recently, however, the Ces\`aro operator has been generalized to a tuple of operators acting on the Drury Arveson space on the unit ball $H^2_d(\mathbb{B}_d)$, which we refer to as the Ces\`aro tuple, in a way that captures the key features of the classical Ces\`aro operator \cite{P}.

In this article we extend this definition to introduce a tuple of Rhaly-type operators acting on the Drury-Arveson space.

\subsection{The Drury-Arveson Space}

Recall that the Drury-Arveson space on the unit ball $H^2_d(\mathbb{B}_d)$ is the reproducing kernel Hilbert space with kernel 

\begin{equation}k(z,w)=k_w(z)=\frac{1}{1-\langle z, w \rangle}\end{equation}

where $\langle \cdot, \cdot \rangle$ is the standard inner product. It is widely viewed as the appropriate generalization of the classical Hardy space on the disk.

As a quick review of notation in several variables. Recall that for $z=(z_1,...,z_d) \in \mathbb{C}^d$, we let 

\begin{equation}z^{\alpha}=\prod_{i=1}^d z_i^{\alpha_i}\end{equation}

for every multi-index $\alpha=(\alpha_1,..., \alpha_d) \in \mathbb{N}^d$. Also we write

\begin{equation}
\alpha !=\prod_{i=1}^d \alpha_i !
 \quad\text{and}\quad
|\alpha|=\sum_{i=1}^d \alpha_i.
\end{equation}

A standard ordering is established using the multi-indices of the monomials $z^{\alpha}$ by simply endowing them with lexographical order. 

For two functions $f,g \in H^2_d$ if we have power series expansions given by

\begin{equation}
f(z)=\sum_{\alpha} c_{\alpha}z^{\alpha}
\quad \text{and} \quad
g(z)=\sum_{\alpha} d_{\alpha}z^{\alpha}
\end{equation}

then we define their inner product to be

\begin{equation}\langle f, g \rangle_{H^2_d}=\langle f, g \rangle=\sum_{\alpha} \frac{\alpha!}{|\alpha|!}c_{\alpha} \overline{d_{\alpha}}.\end{equation}

An orthonormal basis is then given by $\{e_{\alpha}\}$ where

\begin{equation}e_{\alpha}=\sqrt{\frac{|\alpha|!}{\alpha!}}z^{\alpha}.\end{equation}

Notably, the set of monomials $\{z^{\alpha}\}$ form an orthogonal basis for $H^2_d(\mathbb{B}_d)$.

An analytic function $f(z)=\sum_{\alpha}c_{\alpha}z^{\alpha}$ is in $H^2_d$  if

\begin{equation}||f||^2_{\mathbb{H}^2_d}=||f||^2=\sum_{\alpha} |c_{\alpha}|^2\frac{\alpha !}{|\alpha|!}< \infty.\end{equation}

It will often be useful to write $f$ as a decomposition of homogeneous functions $f_n(z)=\sum_{|\alpha|=n} a_{\alpha}z^{\alpha}$ which then gives us $f(z)=\sum_{n=0}^{\infty} f_n(z)$. It immediately follows that $\langle f_n, g_m \rangle=0$ for $m \neq n$.

For further reading on the Drury-Arveson space, the reader is referred to survey articles by M. Hartz \cite{Hartz} and O. Shalit \cite{Shalit} or the original articles by Drury and Arveson \cite{Drury}, \cite{Arveson}. 

\section{The Ces\`aro tuple}

\subsection{Definition}

In \cite{P}, the definition for a Ces\`aro tuple is given as follows:

\begin{Definition}
Let $f=\sum_{n=0}^{\infty}f_n$ be the homogeneous expansion of $f \in H^2_d$. For $\{j\}_{1}^d$, we define a $d$-tuple of operators given by $[C_1,...,C_d]$ where, for $1 \leq j \leq d$,

\begin{equation}
C_jf=f_0+\frac{1}{2}(z_j f_0+f_1)+ \frac{1}{3} (z_j^2 f_0 + z_j f_1 + f_2) + \ldots 
\end{equation}

\end{Definition}

Since the monomials form an orthogonal basis for $H^2_d$ and each $C_j$ is a bounded linear operator, we may equivalently define this operator via its action on the monomials $z^{\alpha}$. 

\begin{Definition}
We define our $d$-tuple of operators given by $[C_1,...,C_d]$ for $1 \leq j \leq d$ via their action on monomials by

\begin{equation}
C_jz^{\alpha}=\sum_{k=0}^{\infty}\frac{1}{|\alpha|+k+1}z^{\alpha+ke_j}.\label{def}
\end{equation}

As demonstrated in \cite{P}, numerous key properties of the classical Ces\`aro operator hold when applied to this tuple, giving support to the notion that it deserves to be consider the appropriate generalization of the classical Ces\`aro operator to the Drury-Arveson space.

While it is shown in \cite{P} that each $C_j$ is bounded, it remains to determine the exact norm of $C_j$. As we will recall, this norm plays a key role in determining our sufficient condition for the boundedness of our generalized Rhaly operators.

\end{Definition}

\subsection{The Norm of $C_j$}

We next explicitly compute the norm $||C_j||$ and find that it coincides with the norm of the classical Ces\`aro operator.

\begin{Theorem}
Given the Ces\`aro Tuple $\{C_j\}$, we have that $||C_j||=2$. \label{C}
\end{Theorem}

\begin{proof}
Note that we can decompose the Drury-Arveson space as a direct sum given by 

\begin{equation}\label{M}
H^2_d(\mathbb{B}_d)=\bigoplus_{\alpha \mid \alpha_j=0}M_{\alpha}
\end{equation}

 where $M_{\alpha}= \overline{\text{span}}\{z^{\alpha+ke_j}, k \geq 0, \alpha_j=0\}$. Since $C_j$ maps each $M_{\alpha}$ into itself, it is block diagonal with respect to the decomposition $\bigoplus_{\alpha \mid \alpha_j=0}M_{\alpha}$ and hence we have that  $||C_j||=\sup_{\alpha} ||C_j|_{\alpha}||$,

Next, let $f(z)=\sum_{\alpha}a_{\alpha}z^{\alpha}$ be in $H^2_d$ and $z^{\alpha'}=\prod_{i \neq j} z_i^{\alpha_i}$. Since $C_j$ is a bounded operator, restricting to $M_{\alpha}$ gives us 

\begin{equation}
C_jf(z)=C_j\left(\sum_{\alpha_j=0}^{\infty} a_{\alpha} z^{\alpha}\right)=\sum_{\alpha_j=0}^{\infty} a_{\alpha}\left(\sum_{k=0}^{\infty}\frac{1}{\alpha_j+|\alpha'|+k+1} z^{\alpha_j+k}\right) \left(z^{\alpha'} \right)
\end{equation}

and hence, since the monomials form an orthogonal basis, we have

\begin{equation}
||C_jf(z)||^2=\sum_{\alpha_j=0}^{\infty}  |a_{\alpha}|^2\left(\sum_{k=0}^{\infty}\frac{1}{(\alpha_j+|\alpha'|+k+1)^2} \frac{(\alpha_j+k)!\alpha'!}{(\alpha_j+\alpha'+k)!}\right)
\end{equation}

which is maximized when $|\alpha'|=0$. Thus for $M_0$, we have 

\begin{equation}
||C_jf(z)||^2=\sum_{\alpha_j=0}^{\infty}  |a_{\alpha}|^2\left(\sum_{k=0}^{\infty}\frac{1}{(\alpha_j+k+1)^2}\right).
\end{equation}

Letting $m=\alpha_j$, denoting $a_{\alpha}=a_m$ and letting $n=m+k$ gives

\begin{equation}
||C_jf(z)||^2=\sum_{m=0}^{\infty}  |a_m|^2\left(\sum_{n=m}^{\infty}\frac{1}{(n+1)^2}\right).
\end{equation}

By the defining properties of the Drury-Arveson space, the Drury-Arveson norm coincides with the Hardy space norm for $|\alpha'|=0$ and the above expression is precisely that of the classical Ces\`aro operator. It well-known that this norm is $2$ and hence we conclude that $||C_j||=2$.

\end{proof}

\section{The Sequence Space Isometrically Isomorphic to $H^2_d$.}

Just as in the case of one variable, one may establish results via a sequence space. We recall the deep connection between $\ell^2$ and $H^2$ in the unit disk and aim to mimic this approach for $H^2_d$ by equipping our standard sequence space with a new norm. For further background on $\ell^2$ spaces, the reader is referred to \cite{Cheng}. Our sequence space will be a weighted $\ell^2$ space. While weighted $\ell^2$ spaces have been previously taken into consideration (see, e.g., \cite{Jameson}), the primary focus has been on weighted Hardy spaces and spaces of analytic function in the unit disk. The author is not aware of previous work on the weighted sequence spaces motivated by function spaces of several variables. 

We first recall some notation. Let $w=\{w_n\}$ be a sequence of positive real numbers. We let $\ell^2(w)$ denote the weighted sequence space $a=\{a_n\}$ such that $||a||_w^2=\sum_{n=0}^{\infty}w_n |a_n|^2<\infty$. Hence $\ell^2$ is the sequence space with $w_n=1$ for all $n$.

Motivated by the Drury-Arveson space, we have the following definition.

\begin{Definition}
Given the standard multivariable ordering, we define $p^2_d$ to be the set of sequences with lexographical ordering ${\bf a}=\{a_{\alpha}\}_{|\alpha|=0}^{\infty}$ with $\alpha=(\alpha_1,...,\alpha_d)$ such that $\sum_{\alpha=0}^{\infty}\frac{\alpha!}{|\alpha|!}|a_{\alpha}|^2<\infty$.
\end{Definition}

We take the convention that the subscript $d$ will be omitted when the argument is not dependent on the number of variables.

Just as in the definition of the Drury-Arveson space, for $a,b \in p^2$, we define the inner product on $p^2$ to be given by $\langle a, b \rangle=\sum_{n=0}^{\infty}\sum_{|\alpha|=n} \frac{\alpha!}{|\alpha|!}a_{\alpha} \overline{b_{\alpha}}$ and thus $||a||=\sqrt{\sum_{\alpha=0}^{\infty}\frac{\alpha!}{|\alpha|!}|a_{\alpha}|^2}$. It's straightforward to observe that this defines an inner product. Additionally, standard functional analysis arguments demonstrate that, since $p^2$ is a weighted $\ell^2$ space with positive real coefficients, it is a complete inner product space and hence a Hilbert space.  In fact, it is straightforward to see that $p^2$ is isometrically isomorphic to $H^2_d$ via $\sigma: H^2_d \rightarrow p^2$ given by the identification $f(z)=\sum_{\alpha} a_{\alpha}z^{\alpha} \in H^2_d \rightarrow \{a_{\alpha}\} \in p^2$.

Thus, when aiming to resolve questions about an operator acting on $H^2_d$, showing our results for $p^2$ will be sufficient.

\section{Rhaly-Type Operators}

\subsection{Definition}

We follow the definition given in \cite{P} to guide our construction. Given a complex sequence $\{a_n\}$, we define a generalized Rhaly-type operator as a tuple of operators $\{R_{(j, a_n)}\}|_{j=1}^d$ acting on $H^2_d(\mathbb{B}_d)$ by

\begin{equation}
R_{(j, a_n)}f=a_0f_0+a_1(z_jf_0+f_1)+a_2(z_j^2f_0+z_jf_1+f_2)+\cdots
\end{equation}

Since the monomials form an orthogonal basis for $H^2_d(\mathbb{B}_d)$, we may also define the Rhaly-type operator via its action on the monomials. We have

\begin{equation}
R_{(j, a_n)}z^{\alpha}=\sum_{k=0}^{\infty} a_k z^{\alpha+ke_j}.
\end{equation}

When $R_{(j, a_n)}$ is bounded, these definitions coincide. Likewise, letting $f(z)=\sum_{k=0}^{\infty}f_k(z)$ be the homogeneous decomposition of $f(z)$, we have

\begin{equation}
R_{(j, a_n)}f=\sum_{n=0}^{\infty}a_n \sum_{m=0}^n z_j^{n-m}f_m(z).
\end{equation}

In the same manner, we may also consider $R_{(j, a_n)}$ as an operator acting on the sequence space $p^2_d$. For properties of boundedness and compactness, results on these spaces coincide and we toggle back and forth between the two as needed.

Next, given a sequence of complex numbers $\{b_n\}$, we denote $D_{(j,b_n)}$ to be the diagonal operator defined by its action on the monomials as follows. Let $\alpha_j$ denote the $j^{\text{th}}$ index in the mult-index notation $z^{\alpha}$. Then we define 

\begin{equation}
D_{(j,b_n)}z^{\alpha}=b_{\alpha_j}z^{\alpha}.
\end{equation}

Since $D_{(j,b_n)}$ is a diagonal operator, it is bounded if $\sup |b_n|<\infty$. Suppose that $D_{(j,b_n)}$ and $R_{(j,a_n)}$ are bounded. One then verifies that for $f(z)=\sum_{\alpha}c_{\alpha}z^{\alpha}$ we have 

\begin{align}
(D_{(j,b_n)}R_{(j,a_n)})f(z)&=D_{(j,b_n)} \left( \sum_{\alpha} c_{\alpha} \left(\sum_{k=0}^{\infty}a_k z^{\alpha+ke_j} \right)\right)\\
&=\sum_{\alpha} c_{\alpha}\sum_{k=0}^{\infty}a_k D_{(j,b_n)}z^{\alpha+ke_j}=\sum_{\alpha} c_{\alpha}\sum_{k=0}^{\infty}a_k b_k z^{\alpha+ke_j}\\
&=(R_{(j,b_na_n)})f(z).
\end{align}

and hence $D_{(j,b_n)}R_{(j,a_n)}=R_{(j,b_na_n)}$. Next, using the convention $\frac{|a_n|}{a_n}=\frac{a_n}{|a_n|}=1$ when $a_n=0$, it immediately follows that

\begin{equation}
D_{(j,\frac{|a_n|}{a_n})}R_{(j, a_n)}=R_{(j, |a_n|)} \qquad D_{(j,\frac{a_n}{|a_n|})}R_{(j, |a_n|)}=R_{(j, a_n)}. \label{D}
\end{equation}

Furthermore it immediately follows, as in the one-dimensional case, that 

\begin{equation}
D_{(j, (n+1)a_n)}C_{j}=R_{(j, a_n)} \qquad \text{and} \qquad D_{(j,\frac{1}{(n+1)a_n})}R_{(j, a_n)}=C_{j}. \label{B}
\end{equation}

\subsection{Boundedness}

In establishing the boundedness conditions of our operator, we follow the strategies of the initial proofs given by Rhaly in the one-dimensional case \cite{R2} and adapt them to the case of several variables.

We began by showing that we may assume our sequence is positive.

\begin{Theorem}
$||R_{(j, a_n)}||=||R_{(j, |a_n|)}||$ and $R_{(j, a_n)}$ is bounded if and only if $R_{(j, |a_n|)}$ is bounded.
\end{Theorem}

\begin{proof}
Suppose $\{a_n\}$ is not identically $0$ so that $||D_{(j, \frac{a_n}{|a_n|})}||=||D_{(j, \frac{|a_n|}{a_n})}||=1$. It follows immediately from our factorizations in \ref{D} that $R_{(j, a_n)}$ is bounded if and only if $R_{(j, |a_n|)}$ is bounded. Furthermore,

\begin{equation}
||R_{(j, |a_n|)}|| \leq ||D_{(j,\frac{|a_n|}{a_n})}|| \cdot ||R_{(j, a_n)}||=||R_{(j, |a_n|)}||
\end{equation}

and 

\begin{equation}
||R_{(j, a_n)}|| \leq ||D_{(j,\frac{a_n}{|a_n|})}|| \cdot ||R_{(j, |a_n|)}||=||R_{(j, |a_n|)}||
\end{equation}

from which we conclude our result.
\end{proof}

We next investigate boundedness. Of course this depends on the sequence $\{a_n\}$ which we may assume is positive. In \cite{R2}, Rhaly showed that the Rhaly operator is bounded if and only if $L$ exists and $L<\infty$ where $L=\lim_{n \rightarrow \infty}(n+1)a_n$ and is compact if $L=0$. We demonstrate the same is true for our generalized Rhaly operator.

\begin{Theorem}
Suppose $L=\lim_{n \rightarrow \infty}(n+1)a_n$ exists. If $L < \infty$, then $R_{(j, a_n)}$ is bounded. If $L=\infty$, then $R_{(j, a_n)}$ is unbounded. If $L=0$, then $R_{(j, a_n)}$ is compact.
\end{Theorem}

\begin{proof}

First we prove $L<\infty$ implies boundedness by explicitly computing an upper bound. By \ref{B}, \ref{C}, and the submultiplicativity of the operator norm we have

\begin{equation}
||R_{(j, a_n)}||  \leq ||D_{(j, (n+1)a_n)}|| \cdot ||C_{j}||= 2 ||D_{(j, (n+1)a_n)}||=2 \sup_{n} |(n+1)a_n|.
\end{equation}

which gives us an explicit upper bound for $R_{(j, a_n)}$. 

We next suppose $L=\infty$. For $R_{(j, a_n)}$ to be bounded means that there exists a constant $M>0$ such that for every $f \in H^2_d$,  $||R_{(j, a_n)}f||^2 \leq M^2 ||f||^2$. However, given $f(z)=g(z_1)$, by definition of $H^2_d$, $g$ is in $H^2(\mathbb{D})$. Additionally, $||f||=||g||_{H^2(\mathbb{D})}$. Since $R_{(j, a_n)}$ is unbounded when restricted to functions in one variable, it is unbounded on a subset of the functions in $H^2_d$ and hence unbounded on $H^2_d$.

Now suppose $L=0$. Thus the sequence $\{(n+1)a_n\}$ must decay to zero. It is sufficient to construct a sequence of finite-rank operators whose operator norm limit is $R_{(j, a_n)}$. Define $R_{(j, a_n)}^{(N)}$ to be the operator whose matrix representation is such that the first $N$ rows match $R_{(j, a_n)}$ and the remaining rows are zero. Clearly $R_{(j, a_n)}^{(N)}$  is compact. Next we show that $R_{(j, a_n)}-R_{(j, a_n)}^{(N)}$ converges to $0$ in the operator norm as $N \rightarrow \infty$. 

We then have $(R_{(j, a_n)}-R_{(j, a_n)}^{(N)})=D_{(j, (n+1)a_n)}^{(N)}C_{j}$ where $D_{(j, (n+1)a_n)}^{(N)}$ is the diagonal operator whose first $N$ diagonal elements are zero and the remaining coincide with $D_{(j, (n+1)a_n)}$. Thus

\begin{equation}
||(R_{(j, a_n)}-R_{(j, a_n)}^{(N)})|| =||D_{(j, (n+1)a_n)}^{(N)}C_{j}|| \leq 2 \sup_{n>N} |(n+1)a_n|
\end{equation}

which decays to zero as $N\rightarrow \infty$ and hence $R_{(j, a_n)}$ is the limit of finite rank operators and is thus compact.

\end{proof}

\subsection{The Adjoint}

It is also shown in \cite{P} that the adjoint of $C_j$ acting on the monomial $z^{\beta}$ is given by

\begin{equation}
C_j^*z^{\beta}=\frac{\beta !}{(|\beta|+1)!}\sum_{k=0}^{\beta_j}\left(\frac{(|\beta|-k)!}{(\beta-ke_j)!} \right) z^{\beta-ke_j}.
\end{equation}

Hence when $L<\infty$, we have 

\begin{equation}
R_{(j, a_n)}^*=\left(D_{(j, (n+1)a_n)}C_{j} \right)^*=D_{(j, (n+1)a_n)}^*C_{j}^*
\end{equation}

which gives 

\begin{equation}
R_{(j, a_n)}^*z^{\alpha}=\frac{\beta !}{(|\beta|+1)!}\sum_{k=0}^{\beta_j}\left(\frac{(|\beta|-k)!}{(\beta-ke_j)!} \right) \overline{(\beta_j-k+1)a_{\beta_j-1}}z^{\beta-ke_j}.
\end{equation}

\subsection{Preliminary Spectral Properties}

We first investigate the point spectrum of $R_{(j, a_n)}$. 

\begin{Theorem}\label{s}
Let $R_{(j, a_n)}$ be a Rhaly-type operator. Then $\sigma_p (R_{(j, a_n)}) \subset \{a_0, a_1, ...\}$.
\end{Theorem}

\begin{proof}

Note that for $R_{(j, a_n)}f=\lambda f$, the $s^{\text{th}}$ homogeneous term is given by $(R_{(j, a_n)})_s=a_s\sum_{m=0}^s z^{s-m}_j f_m(z)$ and thus 

\begin{equation}
a_s\sum_{m=0}^s z^{s-m}_j f_m(z)=\lambda f_s(z).
\end{equation}

We can write this as  

\begin{equation}
a_sf_s+a_s\sum_{m=0}^{s-1} z^{s-m}_j f_m=\lambda f_s \rightarrow (\lambda-a_s)f_s=a_sz_j \sum_{m=0}^{s-1} z_j^{s-1-m}f_m.
\end{equation}

Now, from the $(s-1)^{\text{th}}$ equation we have for $a_{s-1} \neq 0$, 

\begin{equation}
\lambda f_{s-1}=a_{s-1}\sum_{m=0}^{s-1} z^{s-1-m}_j f_m \rightarrow \sum_{m=0}^{s-1} z^{s-1-m}_j f_m=\frac{\lambda}{a_{s-1}}f_{s-1}.
\end{equation}

Substituting this into our previous equation gives 

\begin{equation}
 (\lambda-a_s)f_s=\left(\frac{\lambda a_s}{a_{s-1}}f_{s-1}\right) z_j .
\end{equation}

Suppose $f$ is not identically zero. Then there is a minimal $k \geq 0$ such that $f_k \neq 0$ and $f_j=0$ for all $j <k$. Now it follows that $a_kf_k=\lambda f_k$ and hence $\lambda =a_k$. 
\end{proof}

Additionally, supposing that $a_m \neq0$ for $m>k$, notice that by mathematical induction we obtain

\begin{equation}
 f_{k+n}=\frac{a_k^{n-1}a_{k+n}}{\prod_{m=k+1}^{k+n}(a_k-a_{k+n})}z_j^n f_k.
\end{equation}

Hence $f_s$ for $s>k$ is uniquely determined as a multiple of $z_j^{s-k}$. 

For a given complex sequence $\{a_n\}$, determining whether it actually is an eigenvalue is highly contingent on the sequence. One must first check whether the sequence $\{a_n\}$ is in $p^2_d$. Certainly if there exists a $k\geq 0$ such that $a_n=0$ for $n>k$, i.e. a polynomial, then the point spectrum is a finite set of discrete eigenvalues as described above.

One immediate consequence of \ref{s} is the following:

\begin{Theorem}
If $L=0$ so that $R_{(j, a_n)}$ is compact, then its spectrum is given by $\sigma (R_{(j, a_n)}) \subset \{0\} \cup \{a_0, a_1, ...\}$.
\end{Theorem}

The next step would be to compute the spectrum of $R_{(j, a_n)}$ when $L<\infty$. A useful approach would be to use the decomposition given by \ref{M} and then apply what is known about the spectrum of the Rhaly operator on the Hardy space. They will likely coincide. The author leaves this to the interested reader to demonstrate.

\section{The Roadmap to Future Discovery}

In the introduction, we gave several ways in which the Rhaly operators have been generalized from their action on $\ell^2$. Such generalizations, such as the weighted Hardy spaces, have well-known generalizations in spaces of several variables as well. Triangulating the generalizations in one variable to those in several variables, of course, remains to be done. For example, the Drury-Arveson space is one space on a scale of reproducing kernel Hilbert spaces known as the standard weighted Bergman spaces (or Hardy-Sobolev spaces) given by the reproducing kernels

\begin{equation}
k^{\alpha}(z,w)=\frac{1}{(1-\langle z, w \rangle)^{\alpha}}.
\end{equation}

For $d$-variables, we have $d=1$ the Szego kernel corresponding to the Drury-Arveson space. We have $\alpha=d$ corresponding to the generalized Hardy space and $\alpha=d+1$ corresponding to the generalized Bergman space. Let $\mathcal{O}(\mathbb{B}_d)$ denote the set of function of $d$ complex variables analytic on the unit ball $\mathbb{B}_d$. We may write the definitions in terms of power series as follows:

\begin{equation}
\mathcal{H}_a=\{f(z)=\sum_{\alpha}a_{\alpha}z^{\alpha} \in \mathcal{O}(\mathbb{B}_d) \mid \sum_{\alpha}|a_{\alpha}|^2\frac{\alpha!}{|\alpha|!}(|\alpha|+1)^{1-\alpha}<\infty\}.
\end{equation}

Our Ces\`aro-type operators and Rhaly-type operators have yet to be investigated on these spaces. For more on these spaces, see \cite{Zhao}.

Additionally, properties such as hyponormality have yet to be investigated for the Ces\`aro-type operators. This should be addressed beforing considering the case of the Rhaly-type operators. 

Much work remains to be done. The author leaves it to the enthusiastic reader to enjoy these investigations.

\end{document}